\documentclass[12pt]{amsart}

\usepackage{amsmath}

\usepackage{amssymb}
\usepackage{mathrsfs}
\usepackage{a4wide}
\usepackage[all]{xy}
\usepackage{xcolor}
\usepackage{tikz}

\usetikzlibrary{arrows,positioning,automata,shapes,calc,3d,graphs}
\usetikzlibrary{decorations.markings}
\usetikzlibrary{arrows, shadows}
\tikzset{bendy/.style={bend left=90,looseness=1.5,draw,thick}}
\tikzset{straight/.style={bend left=0,looseness=1.5,draw,thick}}
\tikzset{vert/.style={draw,circle,inner sep=1.5pt,minimum size=6pt}}
\tikzset{dvert/.style={draw,diamond,inner sep=1pt,minimum size=10pt}}
\tikzset{reality/.style={->,draw,thick,red,shorten <=-3pt,shorten >=-3pt}}
\tikzset{desire/.style={bend left=90,looseness=1.5,draw,thick,blue}}

\theoremstyle{plain}
\newtheorem{theorem}{Theorem}[section]
\newtheorem{lemma}[theorem]{Lemma}

\newtheorem{corollary}[theorem]{Corollary}

\newtheorem*{repp@theorem}{\repp@title (reformulated)}
\newcommand{\newrepptheorem}[2]{%
\newenvironment{repp#1}[1]{%
 \def\repp@title{#2 \ref{##1}}%
 \begin{repp@theorem}}%
 {\end{repp@theorem}}}
\makeatother
\newrepptheorem{theorem}{Theorem}  

\theoremstyle{definition}
\newtheorem{definition}[theorem]{Definition}

\makeatletter
\newtheorem*{rep@theorem}{\rep@title \ continued}
\newcommand{\newreptheorem}[2]{%
\newenvironment{rep#1}[1]{%
 \def\rep@title{#2 \ref{##1}}%
 \begin{rep@theorem}}%
 {\end{rep@theorem}}}
\makeatother
\newreptheorem{example}{Example}

\newcommand{\open}{\mathcal{O}}
\newcommand{\dense}{\mathcal{D}}
\newcommand{\Dense}{\mathfrak{O}}
\newcommand{\reals}{\mathbb{R}}
\newcommand{\sone}{{\sf S}_1}
\newcommand{\sfin}{{\sf S}_{fin}}

\newcommand{\gone}{{\sf G}_1}
\newcommand{\gfin}{{\sf G}_{fin}}

\newcommand{\T}{\textsf{T}}

\makeatletter
\@namedef{subjclassname@2010}{%
  \textup{2010} Mathematics Subject Classification}
\makeatother

\title{Selective versions of $\theta$-density} 

\author[Babinkostova]{L. Babinkostova}
\address{Department of Mathematics, Boise State University, Boise, Idaho, U.S.A.}

\author[Pansera]{B.A. Pansera}
\address{Decisions Lab, Mediterranean University of Reggio Calabria, Italy}

\author[Scheepers]{ and M. Scheepers}
\address{Department of Mathematics, Boise State University, Boise, Idaho, U.S.A.}


\begin{document}
\maketitle
\vspace{-0.25in}

\begin{abstract}
In \cite{CasDiMaio} the authors initiate the study of selective versions of the notion of $\theta$-separability  
in non-regular spaces. In this paper we continue this investigation by establishing connections between the familiar cardinal numbers arising in the set theory of the real line, and game-theoretic assertions regarding $\theta$-separability. 
\end{abstract}

\section*{Introduction}

The topological notions 
considered here are related to an operation introduced by Velichko in \cite{Velichko}. This operation is reminiscent of, but is not
, a closure operation. To define the operation, let $(X,\tau)$ be a topological space and let $A$ be a subset of $X$:
\begin{center}
$cl_{\theta}(A) = \{x\in X :$ for every neighborhood $U$ of $x,\; \overline{U}\cap A \neq\emptyset\}$
\end{center}
Here, $\overline{U}$ denotes the closure of $U$ in the topology $\tau$ of $X$. 
The set $cl_{\theta}(A)$ is a closed set in the topology $\tau$. 
Thus the family of sets of the form $cl_{\theta}(A)$ in a topological space is a subset of the collection of closed sets of that space. 
The set $cl_{\theta}(A)$ is said to be the $\theta$-\emph{closure} of $A$ even though the operation $cl_{\theta}$ is not in general idempotent and thus not a conventional closure operator.  
For a set $A\subseteq X$ the inclusion $\overline{A} \subseteq  cl_{\theta}(A)$ holds. 

A subset of $X$ on which the operator $cl_{\theta}$ is idempotent,\emph{ i.e.}, a set $A\subseteq X$ for which $A = cl_{\theta}(A)$, is called $\theta$-\emph{closed}. By our earlier remarks, $\theta$-closed sets are closed in the ambient topology $\tau$. 
The set $\{X\setminus A: A \mbox{ is $\theta$-closed}\}$, a subset of the collection of open sets of $(X,\tau)$, is a topology on $X$, called the $\theta$ topology and denoted $\tau_{\theta}$. The space $(X,\tau)$ is a $\textsf{T}_3$-space if, and only if, $\tau = \tau_{\theta}$. While each space in this paper is assumed to be a $\textsf{T}_1$-space which, unless specified otherwise, has no isolated points, no space is \emph{a priori} assumed to be $\textsf{T}_3$. Thus, for spaces considered in this paper 
$\tau_{\theta}\subset \tau$. 

\section{The featured topological notions}

Although the concepts we consider are defined in 
\cite{CasDiMaio}, we 
introduce them here for the convenience of the reader.
A subset $D$ of the space $(X,\tau)$ is \emph{dense} if $X = \overline{D}$. A space is said to be \emph{separable} if some countable subset is dense. $D\subset X$ is said to be $\theta$-\emph{dense} if $cl_{\theta}(D) = X$. 
A space is $\theta$-\emph{separable} if it has a countable $\theta$-dense subspace. A space is \emph{strongly} $\theta$-\emph{separable} if each $\theta$-dense subset contains a countable subset which is $\theta$-dense.
In the case when a space is $\theta$-separable or separable, a stronger corresponding version of separability might be considered: 
A countable set $C\subseteq X$ is said to be \emph{groupably} dense if there is a partition $C = \bigcup_{n<\omega} C_n$, where each $C_n$ is a finite set, and for each nonempty open set $U$, for all but finitely many $n$, we have $U\cap C_n \neq \emptyset$. The notion of groupably $\theta$-dense is defined analogously: A countable $\theta$-dense set is groupably $\theta$-dense if there is a partition $C = \bigcup_{n<\omega}C_n$
of $C$ into disjoint finite sets $C_n$ such that for each nonempty open set $U$, for all but finitely many $n$ we have $\overline{U}\cap C_n \neq \emptyset$. 

Notions of \emph{countable tightness} correspond to these closure operators:
Let a point $x\in X$ be given. $X$ is said to be \emph{countably tight at} $x$ if for any set $A\subset X\setminus\{x\}$ for which $x\in\overline{A}$ there is a countable set $B\subseteq A$ such that $x\in\overline{B}$. Likewise $X$ is said to be \emph{countably $\theta$-tight at} $x$ if for any set $A\subset X\setminus\{x\}$ for which $x\in\textsf{cl}_{\theta}(A)$ there is a countable set $B\subseteq A$ such that $x\in\textsf{cl}_{\theta}(B)$. 
In the case when a space is countably $\theta$-tight, or countably tight at $x$, a stronger corresponding version of tightness might be considered: 
$X$ is \emph{groupably} countably tight at $x$ if for each $A\subset X\setminus \{x\}$ for which $x\in\overline{A}$, there is countable set $C\subseteq A$ for which there is a partition $C = \bigcup_{n<\omega}C_n$
of $C$ into disjoint finite sets $C_n$ such that for each nonempty open neighborhood $U$ of $x$, for all but finitely many $n$ we have $U\cap C_n \neq \emptyset$. Under these circumstances we say that $C$ \emph{groups tightly at} 
$x$. 
The notion of \emph{groupably} countably $\theta$ tight at $x$ is defined analogously. 

The set $A$ converges to the point $x$ if for each neighborhood $U$ of $x$ the set $A\setminus U$ is finite.  By analogy $A$ $\theta$-converges to $x$ if for each neighborhood $U$ of $x$ the set $A\setminus\overline{U}$ is finite. A space is \emph{Frech\`et-Urysohn} if: For each subset $A$ of $X$, if $x$ is in the closure of $A$, then there is a sequence in $A$ converging to $x$.
The $\theta$ analogue of this notion as follows: For each subset $A$ of $X$, and for each $x$ in the $\theta$-closure of $A$ there is a sequence in $A$ $\theta$-converging to $x$.

The following symbols will denote the families related to the notions just introduced:
\begin{itemize}
\item{$\Dense = \{A\subseteq X: A \mbox{ dense in } X\}$}
\item{$\Dense^{gp} =\{A\subseteq X: A \mbox{ is groupably dense}\}$}
 \item{$\Dense_{\theta} = \{A\subseteq X: A\; \theta\mbox{-dense in }X\}$}
\item{$\Dense_{\theta}^{gp} = \{A\subseteq X: A \mbox{ is groupably $\theta$-dense}\}$} 
\item{$\Omega_x =\{A\subset X\setminus\{x\}:\; x \in\textsf{cl}(A)\}$.}
\item{$\Omega^{gp}_x = \{A\subset X\setminus\{x\}: A \mbox{ groups tightly at } x\}$.}
 \item{$\Omega^{\theta}_x = \{A\subset X\setminus\{x\}:\; x\in \textsf{cl}_{\theta}(A)\}$}
\item{$\Omega^{\theta;\;gp}_x = \{A\subseteq X:\; A \mbox{ groups $\theta$-tightly at }x\}$.}
\item{$\Gamma_x = \{A\subset X\setminus \{x\}: A \mbox{ converges to }x\}$.}
\item{$\Gamma^{\theta}_x = \{A\subset X\setminus \{x\}: A\; \theta-\mbox{converges to }x\}$.}
\end{itemize}
These families are related to each other, and these relationships are depicted in Figure \ref{fig:closurerelations}.

\begin{figure}[h]
\begin{tikzpicture}

[scale=.5]
{
  \node[dvert] (01) at (4.5,2.6) {$\Omega^{\theta}_x$}; 
  \node[dvert] (78) at (1.5,2.6) {$\Omega_x$};
  \node[dvert] (12) at (4.5,0) {$\Omega^{gp,\theta}_x$} ;
  \node[dvert] (67) at (1.5,0) {$\Omega_x^{gp}$}; ;
  \node[dvert] (23) at (4.5,-2.6) {$\Gamma^{\theta}_x$} ;
  \node[dvert] (56) at (1.5,-2.6) {$\Gamma_x$} ;

  \node[dvert] (34) at (-3.5,0.5) {$\Dense$}; 
  \node[dvert] (45) at ( -0.5,0.5) {$\Dense_{\theta}$} ;

  \node[dvert] (89) at (-3.5,-2.2) {$\Dense^{gp}$}; 
  \node[dvert] (90) at ( -0.5,-2.2) {$\Dense^{gp}_{\theta}$} ;

}

{{\path
    (01) edge[thick,<-] (12);}
{\path
    (12) edge[thick,<-] (67);}
{\path
    (01) edge[thick, <-] (78);}
    {\path
    (67) edge[thick,->] (78);}
    {\path
    (23) edge[thick,<-] (56);}
    {\path
    (67) edge[thick,<-] (56);}
    {\path
    (23) edge[thick,->] (12);}

   {\path
    (34) edge[thick,->] (45);}
    {\path
    (34) edge[thick,dashed,->] (78);}
    {\path
    (45) edge[thick,dashed,->] (01);}
    {\path
    (89) edge[thick,->] (34);}
    {\path
    (89) edge[thick,->] (90);}
    {\path
    (90) edge[thick,->] (45);}
    {\path
    (89) edge[thick,dashed,->] (67);}
    {\path
    (90) edge[thick,dashed,->] (12);}
}
\end{tikzpicture}
\caption{An arrow (dashed or solid) from $A$ to $B$ indicates $A$ is contained in $B$. In $\textsf{T}_3$ spaces, but not in $\textsf{T}_{2.5}$ spaces, the left and right faces coincide.}\label{fig:closurerelations}
\end{figure}
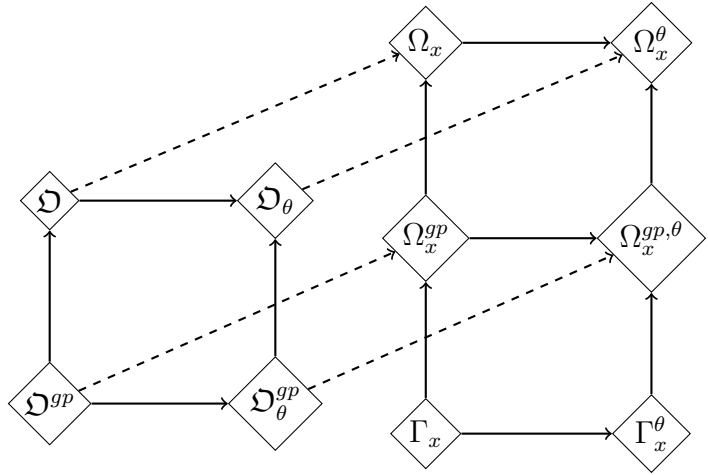

A subset $Y$ of a space $X$ is \emph{discrete} if there is for each $y\in Y$ a neighborhood $U$ of $y$ such that $U\cap Y = \{y\}$. 
There are more families of sets that will be needed in this paper, but these and notation for them will be introduced where needed.

\section{Selection Principles}

We now state the 
selection principles 
featured in this paper:  
For families $\mathcal{A}$ and $\mathcal{B}$ define:
\begin{quote}
$\sone(\mathcal{A},\mathcal{B})$: For each sequence $(A_n:n<\omega)$ of elements of $\mathcal{A}$ there is a sequence $(b_n:n<\omega)$ such that for each $n$  
$b_n\in A_n$ and $\{b_n:n<\omega\}\in\mathcal{B}$.
\end{quote}
\begin{quote}
$\sfin(\mathcal{A},\mathcal{B})$: For each sequence $(A_n:n<\omega)$ of elements of $\mathcal{A}$ there is a sequence $(B_n:n<\omega)$ such that for each $n$ we have $B_n\subseteq A_n$ is finite and $\bigcup\{B_n:n<\omega\}\in\mathcal{B}$.
\end{quote}
These selection principles are anti-monotonic in the first variable and monotonic in the second, and related to each other as depicted in the Figure \ref{fig:monotonicity}. 

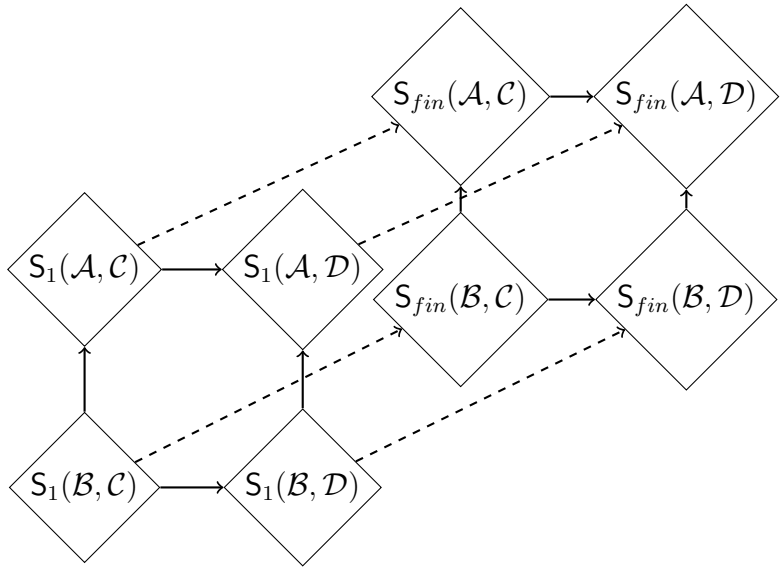
\begin{figure}[h]
\begin{tikzpicture}

[scale=.45]
{
  \node[dvert] (34) at (-0.5,0.5) {$\textsf{S}_{fin}(\mathcal{A},\mathcal{C})$};   
  \node[dvert] (45) at ( 2.5,0.5) {$\textsf{S}_{fin}(\mathcal{A},\mathcal{D})$} ;

  \node[dvert] (89) at (-0.5,-2.2) {$\textsf{S}_{fin}(\mathcal{B},\mathcal{C})$};  
  \node[dvert] (90) at ( 2.5,-2.2) {$\textsf{S}_{fin}(\mathcal{B},\mathcal{D})$} ;
  \node[dvert] (91) at (-5.5,-1.8) {$\textsf{S}_{1}(\mathcal{A},\mathcal{C})$};   
  \node[dvert] (92) at ( -2.6,-1.8) {$\textsf{S}_{1}(\mathcal{A},\mathcal{D})$} ;

  \node[dvert] (93) at (-5.5,-4.7) {$\textsf{S}_{1}(\mathcal{B},\mathcal{C})$};  
  \node[dvert] (94) at ( -2.6,-4.7) {$\textsf{S}_{1}(\mathcal{B},\mathcal{D})$} ;
}

{
    {\path
    (34) edge[thick,->] (45);}
    {\path
    (89) edge[thick,->] (34);}
    {\path
    (89) edge[thick,->] (90);}
    {\path
    (90) edge[thick,->] (45);}

    {\path
    (91) edge[thick,->] (92);}
    {\path
    (93) edge[thick,->] (94);}
    {\path
    (93) edge[thick,->] (91);}
    {\path
    (94) edge[thick,->] (92);}

    {\path
    (91) edge[thick,dashed,->] (34);}
    {\path
    (92) edge[thick,dashed,->] (45);}
    {\path
    (93) edge[thick,dashed,->] (89);}
    {\path
    (94) edge[thick,dashed,->] (90);}

}
\end{tikzpicture}
\caption{An arrow (dashed or solid) from $A$ to $B$ indicates $A$ implies $B$. This figure depicts the relationships when $\mathcal{A}\subset \mathcal{B}$ and $\mathcal{C}\subset\mathcal{D}$.} \label{fig:monotonicity}
\end{figure}

\section{Games}

For given families $\mathcal{A}$ and $\mathcal{B}$ of sets  we consider the following games of length $\omega$ between players ONE and TWO:
\begin{quote}
$\gone(\mathcal{A},\mathcal{B})$: In inning $n<\omega$ ONE chooses $O_n\in\mathcal{A}$, and TWO responds with a $T_n\in O_n$. The play
\[
  O_0,\; T_0,\; O_1,\; T_1,\; \cdots
\] 
is won by TWO if $\{T_n:n<\omega\}\in\mathcal{B}$. Otherwise, ONE wins.
\end{quote}
\begin{quote}
$\gfin(\mathcal{A},\mathcal{B})$: In inning $n<\omega$ ONE chooses $O_n\in\mathcal{A}$, and TWO responds with a finite $T_n\subseteq O_n$. The play
\[
  O_0,\; T_0,\; O_1,\; T_1,\; \cdots
\] 
is won by TWO if $\bigcup\{T_n:n<\omega\}\in\mathcal{B}$. Otherwise, ONE wins.
\end{quote}

If in one of these games player ONE has no winning strategy, then the corresponding selection principle holds. Under special circumstances the converse could also hold. Sections 5 and 6 report on some such circumstances. The selection principle itself rarely implies that TWO has a winning strategy in the game. In Section 7 we report on some circumstances under which player TWO has a winning strategy.

\section{A refined topology for certain $\textsf{T}_3$-spaces}\label{sec:cocountable}

Next, we formalize a construction that has been used in a number of \emph{ad-hoc} examples in the literature. 
For the remainder of this section let $(X,\tau)$ be a topological space in which each nonempty open set is uncountable. Define the finer topology $\tau_c$ generated by basic open sets of form $U\setminus C$ where $U\in\tau$ and $C$ is countable. For a $\tau_c$-basic open set $V$, write $V = U\setminus C$ with $U\in\tau$ and $C\subseteq X$ countable.

\begin{lemma}\label{lemma:openclosure}
The $\tau_c$-closure of $V$ is equal to the $\tau$ closure of $U$.
\end{lemma}
\begin{proof}
 First consider a point $x$ in the $\tau$-closure of $U$. Consider a $\tau_c$ neighborhood $W$ of $x$. Now $W$ is of the form $V\setminus A$ where $V$ is a $\tau$-neighborhood of $x$ and $A$ is a countable set not containing $x$. Since $V\cap U$ is an open set containing $x$ it is uncountable. and $U\cap W$ is uncountable. But then $U\cap V\setminus (A\cup C)$ is an uncountable set contained in $(U\setminus C) \cap (V\setminus A)$.  It follows that $W\cap (U\setminus C)$ is nonempty. Since $W$ was an arbitrary $\tau_c$ neighborhood of $x$, it follows that $x$ is in the $\tau_c$ closure of $V$. Thus, each member of the $\tau$-closure of $U$ is in the $\tau_c$-closure of $V$.

Next, consider any point $y$ in the $\tau_c$-closure of $V$. Any $\tau$-neighborhood of $y$ is also a $\tau_c$ neighborhood of $y$, and thus has uncountable intersection with $U$, implying that $y$ is in the $\tau$-closure of $U$.
\end{proof}

\begin{lemma}\label{relatetheta}
If $(X,\tau)$ is a $\textsf{T}_3$-space,  
then $(\tau_c)_{\theta} = \tau$. 
\end{lemma}
\begin{proof}
We must show that the $\theta$-closed subsets in topological space $(X,\tau_c)$ are the closed subsets of the topological space $(X,\tau)$. Thus, let $A\subset X$ be given. 

Assume that relative to the topology $\tau_c$ we have, for a point $x\in X$ that $x\in cl_{\theta}(A)$. Thus, for such a point $x\in X$ we have for each $\tau_c$-neighborhood $U$ of $x$ that $\overline{U}\cap A\neq \emptyset$. Here $\overline{U}$ is computed relative to the topology $\tau_c$. By Lemma \ref{lemma:openclosure} and the fact that $(X,\tau)$ is $T_3$, it follows that each $\tau$-neighborhood of $x$ has nonempty intersection with $A$, whence $x$ is also in the $\tau$-closure of $A$. Thus, for every set $A\subseteq X$, the $\theta$-closure (in $\tau_c$) of $A$ is a subset of the $\tau$ closure of $A$.

Conversely, let $x$ be an element of the $\tau$-closure of $A$. Thus, for each open neighborhood $U\in\tau$ of $x$, the set $U\cap A$ is nonempty. But as $X$ is $T_3$ and every nonempty open subset is uncountable, for each $\tau$-neighborhood $U$ of $x$, the $\tau$-closure of $U$ meets $A$. But then by Lemma \ref{lemma:openclosure}, for each $\tau_c$ neighborhood $U\setminus C$ (where $C$ is a countable set) of $x$, the $\tau_c$ closure of $U\setminus C$ meets $A$, implying $x$ is a member of $cl_{\theta}(A)$.

Thus, if $A$ is a subset of $X$, then $A$ is $\theta$-closed (relative to the topology $\tau_c$) if, and only if, $A$ is closed in the topology $\tau$.
\end{proof}

The space $C(X)$ is especially suited to the construction above and will be used throughout the several times in this paper.
Let $(X,d)$ be a infinite separable metric space. Then $C(X)$ denotes the set of continuous real-valued functions with domain $X$. The set $C(X)$ is a subset of the product set $\Pi_{x\in X}\reals$. Endowing the latter product with the Tychonoff product topology endows the subset $C(X)$ with the subspace topology otherwise known as the point-wise topology. The symbol $C_p(X)$ denotes the topological space with underlying set $C(X)$, and with the point-wise topology. It is well-known that this space is a completely regular, and thus regular, Hausdorff space in which each nonempty open set is uncountable. The space $C_p(X)$ acts as intermediary between several selection principles. We now recall some of the relevant ones for this paper.

Recall that an open cover $\mathcal{U}$ of a space $X$ is an $\omega$-cover if $X$ is not a member of $\mathcal{U}$, and for each finite subset $F$ of $X$ there is a $U\in\mathcal{U}$ such that $F\subseteq U$. The symbol $\Omega$ is commonly used for the set of $\omega$-covers of $X$. The following two theorems will be used throughout the rest of the paper:
\begin{theorem}\label{thm:cps1}
For a separable metric space $X$ the following statements are equivalent:
\begin{enumerate}
\item{$X$ has the property $\sone(\Omega,\Omega)$}
\item{ONE has no winning strategy in the game $\gone(\Omega,\Omega)$}
\item{$C_p(X)$ has property $\sone(\Omega_{\mathbf 0},\Omega_{\mathbf 0})$}
\item{ONE has no winning strategy in the game $\gone(\Omega_{\mathbf 0},\Omega_{\mathbf 0})$}
\item{$C_p(X)$ has property $\sone(\mathfrak{D},\mathfrak{D})$}
\item{ONE has no winning strategy in the game $\gone(\mathfrak{D},\mathfrak{D})$}
\end{enumerate}
\end{theorem}

In Theorem \ref{thm:cps1} the equivalence of (1) and (3) are proven in \cite{Sakai}, and the equivalence with (2) and (4) are given in \cite{COC2}. The equivalence with (5) and (6) is given in \cite{COC6}.

\begin{theorem}\label{thm:cpsfin}
For a separable metric space $X$ the following statements are equivalent:
\begin{enumerate}
\item{$X$ has the property $\sfin(\Omega,\Omega)$}
\item{ONE has no winning strategy in the game $\gfin(\Omega,\Omega)$}
\item{$C_p(X)$ has property $\sfin(\Omega_{\mathbf 0},\Omega_{\mathbf 0})$}
\item{ONE has no winning strategy in the game $\gfin(\Omega_{\mathbf 0},\Omega_{\mathbf 0})$}
\item{$C_p(X)$ has property $\sfin(\mathfrak{D},\mathfrak{D})$}
\item{ONE has no winning strategy in the game $\gfin(\mathfrak{D},\mathfrak{D})$}
\end{enumerate}
\end{theorem}

In Theorem \ref{thm:cpsfin} the equivalence of (1) and (3) are proven in \cite{AVA1}, and the equivalence with (2) and (4) are given in \cite{COC2}. The equivalence with (5) and (6) is given in \cite{COC6}.

\section{Density, tightness, convergence and cardinal numbers}

Recall that a family $\mathcal{F}$  of nonempty open subsets of a topological space $X$ is a $\pi$-\emph{base} if there is for each nonempty open set $U\subseteq X$ an element $V$ of $\mathcal{F}$ such that $V\subseteq U$. The cardinal number $\pi(X)$ is the minimal cardinality of a $\pi$-base of $X$, and is called the $\pi$-\emph{weight} of $X$. Following \cite{CasDiMaio} we define:
\begin{definition} A family $\mathcal{V}$ of open subsets of a space $X$ is called a $\theta - \pi$-\emph{base}  for $X$ if for any open set $U\subset X$ there is $V\in\mathcal{V}$ such that $V\subset\overline{U}$.
\end{definition}
Any topological space has a $\theta-\pi$-base. 
We define
\[
  \pi_{\theta}(X) = \min\{\vert\mathcal{V}\vert:\; \mathcal{V} \mbox{ is a } \theta -\pi \mbox{ base for }X\}.
\]
As a $\pi$-base for a topological space is also a $\theta-\pi$-base, it follows that $\pi_{\theta}(X) \le \pi(X)$.

\begin{corollary}\label{cor:cocountablepibase}
Let $(X,\tau)$ be a $\textsf{T}_3$-space in which each nonempty open set is uncountable. Then $\pi_{\theta}(X,\tau_c) = \pi(X,\tau)$.
\end{corollary}
\begin{proof}
Note that any $\pi$-base of $(X,\tau)$ is a $\pi_{\theta}$-base of $(X,\tau_c)$. Conversely, if $\mathcal{C}$ is a $\pi_{\theta}$-base for $(X,\tau_c)$, then we may assume that each element is of the form $U\setminus C$ where $U$ is a $\tau$-open set and $C$ a countable set in $X$. But then $\mathcal{B} = \{\textsf{int}(\textsf{cl}(V)): \;V\in \mathcal{C}\}$, where the closure and interior are with respect to the topology $\tau$, is a $\pi$-base for $X$
\end{proof}

\begin{lemma}\label{lemma:density}
Let $(X,\tau)$ be a $\textsf{T}_3$-space in which each nonempty open set is uncountable.  A set $C\subset X$ is $\theta$-dense in $(X,\tau_c)$ if, and only if, it is dense in $(X,\tau)$.
\end{lemma}
\begin{proof}
Let a subset $C$ of $X$ be given.\\
Assume that $\textsf{cl}_{\theta}(C) = X$, in the $\tau_c$-topology. Consider a point $x\in X$, let $U$ be a $\tau$-neighborhood of $x$. For each countable set $A\subset X$ we have $\textsf{cl}_{\tau_c}(U\setminus A) \cap C\neq\emptyset$. However, $\textsf{cl}_{\tau_c}(U\setminus A) = \textsf{cl}_{\tau}(U)$, and thus $\textsf{cl}_{\tau}(U)\cap C \neq\emptyset$. Since this is the case for each $\tau$-neighborhood of $x$, and $(X,\tau)$ is $\textsf{T}_3$, it follows that for each $\tau$-neighborhood $U$ of $X$, $U\cap C\neq\emptyset$. Thus, $\textsf{cl}_{\tau}(C) = X$.

Next assume that $\textsf{cl}_{\tau}(C) = X$. Consider a point $x\in X$, and let $U\setminus A$ be a $\tau_c$-neighborhood of $x$, where now $U$ is $\tau$-open and $A\subset X$ is countable. But for each countable set $A\subset X$ we have $\textsf{cl}_{\tau}(U\setminus A) = \textsf{cl}_{\tau}(U) = \textsf{cl}_{\tau_c}(U\setminus A)$. Therefore $\textsf{cl}_{\tau_c}(U\setminus A)\cap C\neq\emptyset$.
 Since this is the case for each $\tau_c$-neighborhood of $x$, it follows that $\textsf{cl}_{\theta}(C) = X$ in the $\tau_c$-topology.
\end{proof}

\begin{corollary}\label{cor:thetadense}
Let $(X,\tau)$ be a $\textsf{T}_3$-space with all nonempty open sets uncountable.
\begin{enumerate}
\item{$(X,\tau)$ satisfies $\sone(\dense,\dense)$ if, and only if, $(X,\tau_c)$ satisfies $\sone(\dense_{\theta},\dense_{\theta})$.}
\item{$(X,\tau)$ satisfies $\sone(\dense,\dense^{gp})$ if, and only if, $(X,\tau_c)$ satisfies $\sone(\dense_{\theta},\dense_{\theta}^{gp})$.}
\item{$(X,\tau)$ satisfies $\sfin(\dense,\dense)$ if, and only if, $(X,\tau_c)$ satisfies $\sfin(\dense_{\theta},\dense_{\theta})$.}
\item{$(X,\tau)$ satisfies $\sfin(\dense,\dense^{gp})$ if, and only if, $(X,\tau_c)$ satisfies $\sfin(\dense_{\theta},\dense_{\theta}^{gp})$.}
\end{enumerate}
\end{corollary}
The notions $\sone(\dense,\dense^{gp})$ and $\sfin(\dense,\dense^{gp})$ were respectively called GN-separable and H-separable in \cite{BBM}

A topological space is \emph{strongly} separable if each dense subset contains a countable subset that is dense in the space. The well-known cardinal function $\delta$ is defined by
\[
  \delta(X) = \min\{\kappa\ge\aleph_0: \mbox{Each dense $D\subseteq X$ contains a dense $C\subseteq D$ such that }\vert C\vert\le \kappa\}.
\]
The cardinal function $\delta_{\theta}$ is defined by
\[
  \delta_{\theta}(X) = \min\{\kappa\ge\aleph_0: \mbox{Each $\theta$-dense $D\subseteq X$ contains a $\theta$-dense $C\subseteq D$ such that }\vert C\vert\le \kappa\}.
\]
The cardinal function $t_{\theta}(X,x)$ is defined to be the least infinite cardinal $\kappa$ such that each $A\subseteq X\setminus\{x\}$ with $x\in cl_{\theta}(A)$  contains a $C\subseteq A$ such that $x\in cl_{\theta}(C)$ and $\vert C\vert\le \kappa$.

\begin{theorem}\label{cohenreals} For infinite cardinal number $\kappa$ the following are equivalent:
\begin{enumerate}
\item{$\kappa < \textsf{cov}(\mathcal{M})$.}
\item{For each topological space $(X,\tau)$ with $\delta_{\theta}(X) = \aleph_0$ and $\pi_{\theta}(X)\le \kappa$  ONE has no winning strategy in the game  ${\textsf G}_1(\dense_{\theta},\dense_{\theta})$.}
\item{For each topological space $(X,\tau)$ with $\delta_{\theta}(X) = \aleph_0$ and $\pi_{\theta}(X)\le \kappa$, the selection principle  ${\textsf S}_1(\dense_{\theta},\dense_{\theta})$ holds.}
\end{enumerate}
\end{theorem}

\begin{proof} 
{\flushleft{$(1)\Rightarrow (2)$:}} Let $F$ be a strategy for ONE in the game $\gone(\dense_{\theta},\dense_{\theta})$. As the space is strongly $\theta$-dense, we may assume that each move by ONE is a countable $\theta$-dense subset of $X$. Also, let $\{B_{\alpha}:\alpha<\kappa\}$ be a $\pi-\theta$-base for $X$. 

From $F$ we define the following object: For each countable $\theta$-dense set $D\subset X$, fix a bijective enumeration 
\[
  (x_n(D):n<\omega).
\]  
With $D_{\emptyset} = F(\emptyset)$ ONE's first move, we obtain the sequence $(x_{(n)}:n<\omega)$ enumerating $D_{\emptyset}$. For each $n$ we obtain $D_{(n)} = F(x_{(n)})$, enumerated bijectively as $(x_{(n,m)}:m<\omega)$. For each $(n_1,n_2)$ we obtain $D_{(n_1,n_2)} = F(x_{(n_1)},\; x_{(n_1,\; n_2)})$, enumerated bijectively as $(x_{(n_1,\;n_2,\;n)}:n<\omega)$, and so on. Thus, from $F$ we recursively define finite sequences $(D_{\emptyset},\;x_{(n_1)},\;D_{(n_1)},\;x_{(n_1,\;n_2)}, \cdots)$ of partial $F$-plays.

Now fix an $\alpha<\kappa$ and define $D_{\alpha}\subset\;^{\omega}\omega$ as follows:
\[
  D_{\alpha} = \{f\in\;^{\omega}\omega:\; (\forall n<\omega)(x_{(f(0),\cdots,f(n))}\not\in \overline{B}_{\alpha})\}.
\]
{\flushleft{Claim:}} Each $D_{\alpha}$ is a nowhere dense subset of $X$.\\
For let a finite sequence $(n_1,\cdots,n_k)$ of elements of $\omega$ be given. We show that the basic open subset $\lbrack (n_1,\cdots,n_k) \rbrack$ of $^{\omega}\omega$ contains a nonempty open subset disjoint from $D_{\alpha}$.  To see this, note that the set $D_{(n_1,\cdots,\;n_k)}$ is $\theta$-dense in $X$, and so $\overline{B}_{\alpha}\cap D_{(n_1,\cdots,\;n_k)} \neq \emptyset$. Choose an $n$ with $x_{(n_1,\cdots,n_k,n)} \in \overline{B}_{\alpha}$. Then $\lbrack (n_1,\cdots,n_k,n)\rbrack \cap D_{\alpha} = \emptyset$, and $\lbrack (n_1,\cdots,n_k,n)\rbrack \subset \lbrack (n_1,\cdots,n_k)\rbrack$. This completes the proof of the claim.

By $(1)$ the family $\{D_{\alpha}:\alpha<\kappa\}$ of nowhere dense sets do not cover $^{\omega}\omega$. Choose an $f\in\;^{\omega}\omega\setminus \bigcup\{D_{\alpha}:\alpha<\kappa\}$. Then the $F$-play
\[
  F(\emptyset),\; x_{(f(0))}, F(x_{(f(0))}),\; x_{(f(0),f(1))},\cdots
\]
of the game $\gone(\dense_{\theta},\dense_{\theta})$ is lost by ONE.

{\flushleft{$(2)\Rightarrow (3)$}} For let a sequence $(D_n:n<\omega)$ be a given sequence of $\theta$-dense subsets of the space $X$. Define a strategy $F$ for ONE so that for each $n$ ONE's response in inning $n$ is $D_n$. By (2) ONE has no winning strategy in $\gone(\mathfrak{D}_{\theta},\mathfrak{D}_{\theta})$, and thus there is an $F$-play
\[
  D_0,\; x_0,\; D_1,\; x_1,\; \cdots,\; D_n,\; x_n,\; \cdots
\]
lost by ONE. Since for each $n$ we have $x_n\in D_n$, and $\{x_n:n<\omega\}$ is a $\theta$-dense set, (3) follows.

{\flushleft{$(3)\Rightarrow (1)$}}
Assume (3). Let $X$ be a set of real numbers of (infinite) cardinality $\kappa$. Then the space ${\textsf{C}_p}(X)$ under the usual pointwise topology $\tau$ is $T_3$ and has the property that each nonempty open subset is uncountable. This space has $\pi$-weight $\kappa$, and is strongly separable. Thus $(\textsf{C}_p(X),\tau_c)$ has the properties that $\pi_{\theta}(\textsf{C}_p(X)) =\kappa$ (Corollary \ref{cor:cocountablepibase}), and $\delta_{\theta}(\textsf{C}_p(X)) = \aleph_0$ (Lemma \ref{lemma:density}). Applying (3) to this space we find that $(\textsf{C}_p(X),\tau_c)$ has the property $\sone(\mathfrak{D}_{\theta},\mathfrak{D}_{\theta})$. It follows from Corollary \ref{cor:thetadense} that $(\textsf{C}_p(X),\tau)$ has the property $\textsf{S}_1(\mathfrak{D},\mathfrak{D})$. By Theorem 13 of \cite{COC6} the set of real numbers $X$ satisfies $\textsf{S}_1(\Omega,\Omega)$. Since this conclusion holds for any set of real numbers of cardinality $\kappa$ it follows from Theorem 4.8 of \cite{COC2} that $\kappa<\textsf{cov}(\mathcal{M})$. 
\end{proof}

The proof of the following theorem uses similar ideas.
\begin{theorem}\label{dominatingreals1} For infinite cardinal number $\kappa$ the following are equivalent:
\begin{enumerate}
\item{$\kappa < \mathfrak{d}$.}
\item{For each topological space $(X,\tau)$ with $\delta_{\theta}(X) = \aleph_0$ and $\pi_{\theta}(X)\le \kappa$  ONE has no winning strategy in the game  ${\textsf G}_{fin}(\dense_{\theta},\dense_{\theta})$.}
\item{For each topological space $(X,\tau)$ with $\delta_{\theta}(X) = \aleph_0$ and $\pi_{\theta}(X)\le \kappa$, the selection principle  ${\textsf S}_{fin}(\dense_{\theta},\dense_{\theta})$ holds.}
\end{enumerate}
\end{theorem}
\begin{proof}
{\flushleft $(1)\Rightarrow(2)$\footnote{This proof of $(1)\Rightarrow(2)$ is based on an idea in the proof of Theorem 40 of \cite{COC6}. The details were not given explicitly in \cite{COC6}, whence we give details for the argument here}}  .
Observe that a space satisfying (1) is strongly $\theta$-dense: For let $D$ be a given $\theta$-dense subset of the space. Consider the strategy for ONE that calls on ONE to play $D$ in each inning. By (1) this is not a winning strategy for ONE. Let $(D,\; x_0,\; D,\; x_1,\; \cdots,\; D,\; x_n,\; \cdots)$ be a play lost by ONE. Then the countable subset $\{x_n:\; n<\omega\}$ of $D$ is a $\theta$-dense subset of the space.

Now let $\sigma$ be a strategy for ONE in the game $\gfin(\dense_{\theta},\dense_{\theta})$. As the space is strongly $\theta$-dense, we may assume that each move by ONE is a countable $\theta$-dense subset of $X$. Also, let $\{B_{\alpha}:\alpha<\kappa\}$ be a $\pi$-$\theta$-base for $X$. 
From $\sigma$ we define the following object: For each countable $\theta$-dense set $D\subset X$, fix bijective enumeration 
 $(x_{(n)}(D):n<\omega)$ 
where $D = \{x_{(n)}(D): n<\omega\}$.

With $D_{\emptyset} = \sigma(\emptyset)$ ONE's first move, for each finite subset $F$ of $\omega$ enumerate the set 
$D_{F}=\sigma(\{d_{(n)}:n\in F\})$ bijectively as $(d_{(F,n)}:n<\omega)$. For each pair $(F_1,F_2)$ of finite subsets of $\omega$ 
enumerate   
$D_{F_1,F_2} = \sigma(\{d_{(n)}:n \in F_1\}, \{d_{(F_1,n)}: n \in F_2\})$ bijectively as $(d_{(F_1,F_2,n)}:n<\omega)$. In general, for each finite sequence $(F_1, F_2, \cdots,F_k)$ of finite subsets of $\omega$  
enumerate 
$D_{F_1,\cdots,F_k} = \sigma(\{d_{(n)}:n \in F_1\}, \{d_{(F_1,n)}: n \in F_2\},\cdots,\{d_{(F_1,\cdots,F_{k-1},n)}: n\in F_k\})$
 bijectively as $(d_{(F_1,\cdots,F_k,n)}:n<\omega)$.

Now fix an $\alpha<\kappa$ and define $ \Psi(\alpha)$ to be the set of finite sequences $(F_1,\cdots,F_k)$ of finite subsets of $\omega$ 
for which
\[
 \{d_{(F_1,\cdots,F_{k-1},\;j)}: j\in F_k \}
 \cap \overline{B_{\alpha}} \neq \emptyset
\]   
Then each $\Psi(\alpha)$ is a countably infinite subset of the set $^{<\omega}(\lbrack \omega \rbrack^{<\aleph_0})$. If $\{0,\; 1\}$ is endowed with the discrete topology, then the set $^{^{<\omega}(\lbrack \omega \rbrack^{<\aleph_0})}\{0,\;1\}$ is homeomorphic to the Cantor set. Let $X_{\mathcal{B}}$ be the subset of this space consisting of the characteristic functions of the sets $\psi(\alpha)$.

Since $\vert X_{\mathcal{B}}\vert \le \kappa<\mathfrak{d}$, the space $X_{\mathcal{B}}$ is a Menger space, and thus 
ONE has no winning strategy in the game $\gfin(\open,\open)$ on $X_{\mathcal{B}}$ - \cite{WH}. Define the strategy $\gamma$ for ONE in the game $\gfin(\open,\open)$ on $X_{\mathcal{B}}$ as follows:

$\gamma(X_{\mathcal{B}}) = \{\lbrack ( (H),1)\rbrack : H \mbox{ a finite subset of } \omega\}$\footnote{The notation $\lbrack ((F_1,\cdots,F_n),i)\rbrack$ denotes the set of all elements of $^{^{<\omega}(\lbrack \omega \rbrack^{<\aleph_0})}\{0,\;1\}$ that take the value $i$ at the domain element $(F_1,\cdots,F_n)$. }. 
Since $D_{\emptyset} = \sigma(\emptyset)$ is a $\theta$-dense subset of $X$, for each $B\in \mathcal{B}$ we have $\overline{B}\cap \sigma(\emptyset) \neq \emptyset$. It follows that for each $\alpha<\kappa$ there is a finite subset $F$ of $\sigma(\emptyset)$ such that $F = \{d_{(j)}:j\in H\}\cap \overline{B}_{\alpha}\neq \emptyset$, and so the sequence $(H)$ is a member of $\psi(\alpha)$. It follows that the characteristic function of $\psi(\alpha)$ is a member of the basic open set $\lbrack (H,1)\rbrack$, whence $\gamma(X_{\mathcal{B}})$ is an open cover of the space $X_{\mathcal{B}}$. When TWO chooses a finite subset $\mathcal{V}_1\subseteq \gamma(X_{\mathcal{B}})$, say $\mathcal{V}_1 = \{ \lbrack (G_j,1)\rbrack: j\le k\}$, then define 
\begin{enumerate}
\item{$F_1 = \{d_{(i)}\in \sigma(\emptyset): i \in H_1 = \cup\{G_j:j\le k\}$.}
\item{$D_1 = \sigma(F_1) = \{d_{(H_1,j)}: j<\omega\}$}
\item{$\gamma(\mathcal{V}_1) = \{\lbrack ((H_1,H),1)\rbrack: H \mbox{ a finite subset of $\omega$}\}$}
\end{enumerate}

Since $D_1$ is a $\theta$-dense subset of $X$, there is for each $\alpha<\kappa$ a finite subset $G$ of $\omega$ such that $\{d_{(H_1,j)}:j\in G\} \cap \overline{B}_{\alpha}$ is nonempty. It follows that $\gamma(\mathcal{V}_1)$ is an open cover of $X_{\mathcal{B}}$.
When TWO chooses a finite subset $\mathcal{V}_2\subseteq \gamma(\mathcal{V}_1)$, say $\{ \lbrack (H_1,G_j,1)\rbrack: j\le k\}$, then define
\begin{enumerate}
\item{$F_2 = \{d_{(H_1,i)}\in \sigma(F_1): i \in H_2 = \cup\{G_j:j\le k\}$.}
\item{$D_2 = \sigma(F_1,F_2) = \{d_{(H_1,H_2,j)}: j<\omega\}$}
\item{$\gamma(\mathcal{V}_1,\mathcal{V}_2) = \{\lbrack ((H_1,H_2,H),1)\rbrack: H \mbox{ a finite subset of $\omega$}\}$}
\end{enumerate}

Since $D_2$ is a $\theta$-dense subset of $X$, there is for each $\alpha<\kappa$ a finite subset $G$ of $\omega$ such that $\{d_{(H_1,H_2,j)}:j\in G\} \cap \overline{B}_{\alpha}$ is nonempty. It follows that $\gamma(\mathcal{V}_1,\mathcal{V}_2)$ is an open cover of $X_{\mathcal{B}}$.
When TWO chooses a finite subset $\mathcal{V}_3\subseteq \gamma(\mathcal{V}_1,\mathcal{V}_2)$, say $\{ \lbrack ((H_1,H_2,G_j),1)\rbrack: j\le k\}$, then define
\begin{enumerate}
\item{$F_3 = \{d_{(H_1,H_2,i)}\in \sigma(F_1,F_2): i \in H_3 = \cup\{G_j:j\le k\}$.}
\item{$D_3 = \sigma(F_1,F_2,F_3) = \{d_{(H_1,H_2,H_3,j)}: j<\omega\}$}
\item{$\gamma(\mathcal{V}_1,\mathcal{V}_2,\mathcal{V}_3) = \{\lbrack ((H_1,H_2,H_3,H),1)\rbrack: H \mbox{ a finite subset of $\omega$}\}$}
\end{enumerate}
 
 This procedure defines a strategy $\gamma$ for ONE of the game $\gfin(\open,\open)$ on the space $X_{\mathcal{B}}$. Since $X_{\mathcal{B}}$ is a Menger space it follows that $\gamma$ is not a winning strategy for ONE. Consider a $\gamma$-play
 \[
   \mathcal{U}_1,\; \mathcal{V}_1,\;    \mathcal{U}_2,\; \mathcal{V}_2,\;    \mathcal{U}_3,\; \mathcal{V}_3,\;  \cdots
 \]
 lost by ONE of the game $\gfin(\open,\open)$. By the definition of $\gamma$ we find associated sequences 
 \begin{itemize}
 \item[(a)]{$(H_n:n\in{\mathbb N})$ of finite subsets of $\omega$,}
 \item[(b)]{$(F_n:n\in{\mathbb N})$ of finite subsets of the space $X$,}
 \item[(c)]{$(D_n:n\in{\mathbb N})$ of $\theta$-dense subsets of the space $X$}
 \end{itemize} 
 such that: $F_1$ is a finite subset of $D_{\emptyset} = \sigma(\emptyset)$, and 
 \[
   D_{\emptyset}, F_1, D_1,\cdots, F_n, D_n,\cdots
 \]
 is a $\sigma$-play of $\gfin(\mathfrak{D}_{\theta},\mathfrak{D}_{\theta})$.
 
 $F_1 = \{d_{(i)}: i\in H_1\} \subset D_{\emptyset}$, and $F_{n+1} = \{d_{(H_1,\cdots,H_n,i)}: i \in H_{n+1}\}\subseteq D_n$ for all $n$.
 
 We claim that this $\sigma$-play of $\gone(\mathfrak{D}_{\theta},\mathfrak{D}_{\theta})$ is won by TWO. 
 Indeed, consider the set 
 \[
   D = \bigcup_{n=1}^{\infty}F_n.
 \]
Let $B_{\alpha}$ be an element of $\mathcal{B}$. Since $\bigcup_{n=1}^{\infty}\mathcal{V}_n$ is a cover of $X_{\mathcal{B}}$,  $\psi(\alpha)$ is an element of a set of the form $\lbrack ((H_1,\cdots,H_n,G_j),1)\rbrack$ for some finite subset $G_j$ of $H_{n+1}$, and thus (see the definition of $\psi(\alpha)$) an element of the set $\lbrack ((H_1,\cdots,H_n,H_{n+1}),1)\rbrack$. This implies that $F_{n+1}\cap\overline{B}_{\alpha}$ is nonempty, and thus $D$ is $\theta$-dense in $X$.
 
{\flushleft $(2)\Rightarrow(3)$} Let a sequence $(D_n:n<\omega)$ of $\theta$-dense subsets of $X$ be given. Define a strategy $F$ for ONE in the game $\gfin(\mathfrak{D}_{\theta},\mathfrak{D}_{\theta})$ that for each $n$, in the $n$-th inning calls on ONE to select $D_n$. By (2) this strategy is not a winning strategy. Thus, for each $n$ select a finite set $F_n\subset D_n$ witnessing that the play 
\[
 D_0,\; F_0,\; \cdots,\; D_n,\; F_n,\;\cdots
\]
is lost by ONE. Then $\bigcup_{n<\omega}F_n$ is $\theta$-dense in $X$. It follows that $X$ satisfies the selection principle $\sfin(\mathfrak{D}_{\theta}, \mathfrak{D}_{\theta})$.

{\flushleft $(3)\Rightarrow(1)$} Let $X$ be a set of real numbers of the infinite cardinality $\kappa$. Then the space $\textsf{C}_p(X)$ is a completely regular Hausdorff space, thus $\textsf{T}_3$, and each nonempty open subset is uncountable. Moreover $\delta(C_p(X)) = \aleph_0$ and $\pi_{\theta}(\textsf{C}_p(X),\tau_c) = \aleph_0$. It follows that $\delta_{\theta}(C_p(X)\tau_c) = \aleph_0$ and thus, assuming (3), $(\textsf{C}_p(X),\tau_c)$ satisfies the selection property $\sfin(\mathfrak{D}_{\theta},\mathfrak{D}_{\theta})$.  But then by Corollary \ref{cor:thetadense} it follows that $\textsf{C}_p(X)$ satisfies the selection principle $\sfin(\mathfrak{D},\mathfrak{D})$. By Theorem 35 of \cite{COC6}, the space $X$ of real numbers has the selection property $\sfin(\Omega,\Omega)$. It follows that each set of real numbers of cardinality $\kappa$  has the selection property $\sfin(\Omega,\Omega)$. By Theorem 4.6 of \cite{COC2}, $\kappa<\mathfrak{d}$.
\end{proof}

The proofs of the following theorems are omitted as they use ideas similar to the ones in the proofs of Theorems \ref{cohenreals} and \ref{dominatingreals1}.
\begin{theorem}\label{unboundedreals1} For infinite cardinal number $\kappa$ the following are equivalent:
\begin{enumerate}
\item{$\kappa < \mathfrak{b}$.}
\item{For each topological space $(X,\tau)$ with $\delta_{\theta}(X) = \aleph_0$ and $\pi_{\theta}(X)\le \kappa$  ONE has no winning strategy in the game  ${\textsf G}_{fin}(\dense_{\theta},\dense^{gp}_{\theta})$.}
\item{For each topological space $(X,\tau)$ with $\delta_{\theta}(X) = \aleph_0$ and $\pi_{\theta}(X)\le \kappa$, the selection principle  ${\textsf S}_{fin}(\dense_{\theta},\dense^{gp}_{\theta})$ holds.}
\end{enumerate}
\end{theorem}

\begin{theorem}\label{cohenboundedreals} For infinite cardinal number $\kappa$ the following are equivalent:
\begin{enumerate}
\item{$\kappa < \textsf{add}(\mathcal{M})$.}
\item{For each topological space $(X,\tau)$ with $\delta_{\theta}(X) = \aleph_0$ and $\pi_{\theta}(X)\le \kappa$  ONE has no winning strategy in the game  ${\textsf G}_1(\dense_{\theta},\dense^{gp}_{\theta})$.}
\item{For each topological space $(X,\tau)$ with $\delta_{\theta}(X) = \aleph_0$ and $\pi_{\theta}(X)\le \kappa$, the selection principle  ${\textsf S}_1(\dense_{\theta},\dense^{gp}_{\theta})$ holds.}
\end{enumerate}
\end{theorem}

Since the techniques for proof contain no new elements we report results on cardinal numbers related to the combinatorics of the real line and their relevance to $\theta$ versions of tightness and of the Frechet-Urysohn property
without proof. 
Lemma \ref{lem:tightness} records a core fact used in these proofs: 
\begin{lemma}\label{lem:tightness}
Let $(X,\tau)$ be a $\textsf{T}_3$-space in which each nonempty open set is uncountable. For a subset $C$ of $X$ and a point $x\in X\setminus C$, $x$ is in the $\theta$-closure of $C$ in topology $\tau_c$ if, and only if, $x$ is in the $\tau$-closure of $C$.
\end{lemma}

In particular, Lemma \ref{lem:tightness} implies:
\begin{corollary}\label{cor:tightness}
Let $(X,\tau)$ be a $\textsf{T}_3$-space with all nonempty open sets uncountable.
\begin{enumerate}
\item{$(X,\tau)$ satisfies $\sone(\Omega_x,\Omega_x)$ if, and only if, $(X,\tau_c)$ satisfies $\sone(\Omega_x^{\theta},\Omega_x^{\theta})$.}
\item{$(X,\tau)$ satisfies $\sfin(\Omega_x,\Omega_x)$ if, and only if, $(X,\tau_c)$ satisfies $\sfin(\Omega_x^{\theta},\Omega_x^{\theta})$.}
\end{enumerate}
\end{corollary}

As mentioned, 
the proofs of the upcoming theorems are with some adjustments very similar to those of the previous results. 
We leave the details to the reader.

\begin{theorem}\label{cohenstuff} For infinite cardinal number $\kappa$ the following are equivalent:
\begin{enumerate}
\item{$\kappa < \textsf{cov}(\mathcal{M})$.}
\item{For each topological space $(X,\tau)$ and point $x\in X$ with $t_{\theta}(X,x) = \aleph_0$ and $\chi_{\theta}(X,x)\le \kappa$  ONE has no winning strategy in the game  ${\textsf G}_{1}(\Omega^{\theta}_x,\Omega^{\theta}_x)$.}
\item{For each topological space $(X,\tau)$ and point $x\in X$ with $t_{\theta}(X,x) = \aleph_0$ and $\chi_{\theta}(X,x)\le \kappa$, the selection principle  ${\textsf S}_{1}(\Omega^{\theta}_x,\Omega^{\theta}_x)$ holds.}
\end{enumerate}
\end{theorem}

\begin{theorem}\label{boundedstuff} For infinite cardinal number $\kappa$ the following are equivalent:
\begin{enumerate}
\item{$\kappa < \mathfrak{b}$.}
\item{For each topological space $(X,\tau)$ and point $x\in X$ with $t_{\theta}(X,x) = \aleph_0$ and $\chi_{\theta}(X,x)\le \kappa$  ONE has no winning strategy in the game  ${\textsf G}_{fin}(\Omega^{\theta}_x,\Omega^{\theta,gp}_x)$.}
\item{For each topological space $(X,\tau)$ and point $x\in X$ with $t_{\theta}(X,x) = \aleph_0$ and $\chi_{\theta}(X,x)\le \kappa$, the selection principle  ${\textsf S}_{fin}(\Omega^{\theta}_x,\Omega^{\theta,gp}_x)$ holds.}
\end{enumerate}
\end{theorem}

\begin{theorem}\label{boundedcohenstuff} For infinite cardinal number $\kappa$ the following are equivalent:
\begin{enumerate}
\item{$\kappa < \textsf{add}({\mathcal{M}})$.}
\item{For each topological space $(X,\tau)$ and point $x\in X$ with $t_{\theta}(X,x) = \aleph_0$ and $\chi_{\theta}(X,x)\le \kappa$  ONE has no winning strategy in the game  ${\textsf G}_{1}(\Omega^{\theta}_x,\Omega^{\theta,gp}_x)$.}
\item{For each topological space $(X,\tau)$ and point $x\in X$ with $t_{\theta}(X,x) = \aleph_0$ and $\chi_{\theta}(X,x)\le \kappa$, the selection principle  ${\textsf S}_{1}(\Omega^{\theta}_x,\Omega^{\theta,gp}_x)$ holds.}
\end{enumerate}
\end{theorem}

\begin{theorem}\label{dominatingstuff} For infinite cardinal number $\kappa$ the following are equivalent:
\begin{enumerate}
\item{$\kappa < \mathfrak{d}$.}
\item{For each topological space $(X,\tau)$ and point $x\in X$ with $t_{\theta}(X,x) = \aleph_0$ and $\chi_{\theta}(X,x)\le \kappa$  ONE has no winning strategy in the game  ${\textsf G}_{fin}(\Omega^{\theta}_x,\Omega^{\theta}_x)$.}
\item{For each topological space $(X,\tau)$ and point $x\in X$ with $t_{\theta}(X,x) = \aleph_0$ and $\chi_{\theta}(X,x)\le \kappa$, the selection principle  ${\textsf S}_{fin}(\Omega^{\theta}_x,\Omega^{\theta}_x)$ holds.}
\end{enumerate}
\end{theorem}

Expanding
to convergence properties we introduce the following notions: For a point $x$ of the space $X$, a family $\mathcal{N}$ of nonempty open sets is a neighborhood base for $x$ if for each nonempty open $U$ set with $x\in U$ there is an $N\in\mathcal{N}$ with $x\in N\subseteq U$. The minimal cardinality of a neighborhood base for $x$ is denoted $\chi(X,x)$, and is said to be the \emph{character} of $X$ at $x$. By analogy, a family $\mathcal{N}$ of nonempty open subsets of $X$ is said to be a $\theta$ \emph{neighborhood base} for $X$ at $x$ if  for each nonempty open $U$ set with $x\in U$ there is an $N\in\mathcal{N}$ with $x\in N\subseteq \overline{U}$. We define
\[
  \chi_{\theta}(X,x) = \min\{\vert\mathcal{N}\vert:\; \mathcal{N} \mbox{ is a } \theta - \mbox{neighborhood base for }X \mbox{ at }x\}.
\]
\begin{corollary}\label{cor:cocountablechi}
Let $(X,\tau)$ be a $\textsf{T}_3$-space in which each nonempty open set is uncountable. 
Then $\chi_{\theta}((X,\tau_c),x) = \chi((X,\tau),x)$.
\end{corollary}

\begin{lemma}\label{lem:FU}
Let $(X,\tau)$ be a $\textsf{T}_3$-space in which each nonempty open set is uncountable. For a subset $A$ of $X$ and a point $x\in X$, a sequence in $A$ $\theta$ converges to $x$ in the $\tau_c$ topology if, and only if, it converges to $x$ in the $\tau$ topology.
\end{lemma}

\begin{corollary}\label{cor:FU} 
Let $(X,\tau)$ be a $\textsf{T}_3$-space with all nonempty open sets uncountable.
\begin{enumerate}
\item{$(X,\tau)$ satisfies $\sone(\Omega_x,\Gamma_x)$ if, and only if, $(X,\tau_c)$ satisfies $\sone(\Omega_x^{\theta},\Gamma_x^{\theta})$.}
\item{$(X,\tau)$ satisfies $\sfin(\Omega_x,\Gamma_x)$ if, and only if, $(X,\tau_c)$ satisfies $\sfin(\Omega_x^{\theta},\Gamma_x^{\theta})$.}
\end{enumerate}
\end{corollary}

\begin{theorem}\label{gammastuff} For infinite cardinal number $\kappa$ the following are equivalent:
\begin{enumerate}
\item{$\kappa < \mathfrak{p}$.}
\item{For each topological space $(X,\tau)$ and point $x\in X$ with $t_{\theta}(X,x) = \aleph_0$ and $\chi_{\theta}(X,x)\le \kappa$  ONE has no winning strategy in the game  ${\textsf G}_{1}(\Omega^{\theta}_x,\Gamma^{\theta}_x)$.}
\item{For each topological space $(X,\tau)$ and point $x\in X$ with $t_{\theta}(X,x) = \aleph_0$ and $\chi_{\theta}(X,x)\le \kappa$, the selection principle  ${\textsf S}_{1}(\Omega^{\theta}_x,\Gamma^{\theta}_x)$ holds.}
\end{enumerate}
\end{theorem}

\section{When does player TWO win $\gone(\mathcal{A},\mathcal{B})$?}

Our results in the earlier sections address combinatorial conditions equivalent to player ONE of the game in question not having a winning strategy. Establishing similar combinatorial conditions equivalent to player TWO of these games having a winning strategy would provide useful criteria for deciding whether a given instance of one of these games is determined (i.e., a player has a winning strategy), and for identifying the undetermined instances of these games. The following theorem is a model of the type of result we are aiming at.

\begin{theorem}\cite{COC6}\label{thm:coc6}\rm\; For a $\T_3$ space $X$ the following are equivalent:
\begin{enumerate}
\item TWO has a winning strategy in the game $\gone(\Dense,\Dense)$
\item $\pi(X)=\aleph_0$
\end{enumerate}
\end{theorem}

In light of Lemma \ref{relatetheta} we obtain directly from Theorem \ref{thm:coc6} the following extension:
\begin{theorem}\label{specialtheta}
Let $(X,\tau)$ be a $T_3$-space in which each nonempty open set is uncountable. Then the following are equivalent for the space $(X,\tau_c)$:
\begin{enumerate}
\item TWO has a winning strategy in the game $\gone(\Dense_{\theta},\Dense_{\theta})$
\item $\pi_{\theta}(X)=\aleph_0$
\end{enumerate}
\end{theorem}

\bigskip

It is not clear to what extent the equivalences of Theorem \ref{specialtheta} holds for general non-regular topological spaces. 
The following theorem improves Proposition 2.4 of \cite{CasDiMaio}, as well as Theorem \ref{thm:coc6} from \cite{COC6} (Theorem \ref{thm:coc6} follows from Theorem \ref{Separable} by adding the hypothesis that the space $X$ is a $T_3$-space). 
\begin{theorem}\label{Separable} \rm Let $X$ be a topological space which has no finite nonempty open sets. Consider the following statements:
\begin{enumerate}
\item[(1)] $\pi_{\theta}(X) = \aleph_0$. 
\item[(2)] TWO has a winning strategy in the game $\gone(\Dense_{\theta},\Dense^{gp}_{\theta})$
\item[(3)] TWO has a winning strategy in the game $\gone(\Dense_{\theta},\Dense_{\theta})$
\item[(4)] TWO has a winning strategy in the game $\gone(\Dense,\Dense_\theta)$
\end{enumerate} 
 If for each $\theta$-dense set $D$ and each nonempty open set $U$ the set $\overline{U}\cap D$ is infinite, then (1) implies (2); (2) implies (3); 
(3) implies(4) and
if $X$ is separable, then also (4) implies (1).
\end{theorem}

\begin{proof} {\flushleft{\underline{$(1)\Rightarrow(2)$:}}}
 Let $\mathcal{F}$ be a countable $\theta-\pi$-base of $X$. Fix an initial enumeration $(F_n:n<\infty)$ of $\mathcal{F}$. Then define a new enumeration $(G_n:n<\omega)$ of $\mathcal{F}$ such that if $T_n$ denotes the $n$-th triangular number\footnote{For convenience we declare $T_0 = 0$.}, then $(G_{T_n+1},\cdots G_{T_{n+1}})$ is $(F_1,\cdots,F_{n+1})$ for each $n$. 

Now player TWO's strategy is as follows: When in the $n$-th inning ONE chooses the $\theta$-dense subset $D_n$ of $X$, then TWO responds with a point $x_n\in D_n\cap \overline{G}_n\setminus \{x_j:j<n\}$. This is possible since $D_n\cap\overline{G}_n$ is infinite. To show that this strategy is winning for TWO, it suffices to show that the set $E=\{x_n:n<\infty\}$ is groupably $\theta$-dense: For each $n$, define $E_n = \{x_j: T_{n-1} < j\le T_n\}$. The finite sets $E_n$ are disjoint from each other. Let $U$ be a nonempty open subset of $X$. We must show that for all but finitely many $n$, $\overline{U}\cap E_n\neq\emptyset$. Choose the least $n$ with $F_n\subseteq\overline{U}$. As $F_n\subseteq \overline{U}$, it follows that each $x_m$ selected from $\overline{F}_n$ is an element of $\overline{U}\cap E$. Thus, for each $m\ge n$, $\overline{U}\cap E_m$ is nonempty. 

{\flushleft{\underline{$(2)\Rightarrow(3)$:}}} 
This implication is easy to prove.

{\flushleft{\underline{$(3)\Rightarrow(4)$:}}} This implication follows directly from the fact that $\mathfrak{D} \subseteq \mathfrak{D}_{\theta}$.

{\flushleft{\underline{$(4)\Rightarrow(1)$:}}} Now we are assuming that $X$ is a separable space. Let $\sigma$ be a winning strategy for TWO in the game $\gone(\mathfrak{D},\mathfrak{D}_{\theta})$.

{\flushleft{\bf Claim 1:}} For every sequence $(D_1,\cdots,D_n)$ of dense sets there is a nonempty open set $U$ such that for every $x\in U$ there is $D\in{\frak D}$ such that $x=\sigma(D_1,\cdots,D_n,D)$.\\
{\tt Proof of Claim 1: }
Suppose the contrary. Let $(D_1,\cdots, D_n)$ witness that. For each nonempty open set $U\subset X$ choose $x_U\in U$
so that for any dense $D\subset X$, $x_U\neq\sigma(D_1,\cdots,D_n,D)$.
Put $E=\{x_U:$ $U$ a nonempty open set of $X\}$. Then $E$ is dense in $X$ and $\sigma(D_1,\cdots, D_n,E)\in E$ contradicting the selection of elements of $E$. This proves Claim 1.

Start with the empty sequence in $^{<\omega}{\frak D}$. Choose a nonempty open set $U_\emptyset\neq\emptyset$ as Claim 1. Since $X$ is separable, choose a countable set $C_\emptyset\subset U_\emptyset$ dense in $U_\emptyset$, say $C_\emptyset=(x_{(n)};n<\infty)$. For each $n$ choose a dense $D_{(n)}\subset X$ such that $x_{(n)}=\sigma(D_{(n)})$.

By Claim 1 choose for each $n<\infty$ a nonempty open set $U_{(n)}\subset X$ such that for very $x\in U_{(n)}$ there is $D\in{\frak D}$ such that $x=\sigma(D_{(n)},D)$. In each $U_{(n)}$ choose a countable dense $C_{(n)}\subset U_{(n)}$, and enumerate it as $\{x_{(n,k)}:k<\infty\}$. For each $(n,k)$ choose $D_{(n,k)}\in{\frak D}$ with $x_{(n,k)}=\sigma(D_{(n)}, D_{(n.k)})$. 

In general, with $x_\mu, U_\mu, C_\mu, D_\mu$ defined for all $\mu$ of length at most $k$ in $^{<\omega}\omega$, consider any such $\mu = (n_1,\cdots,n_k)$. 
As $(D_{(n_1)},\cdots, D_{(n_1,\cdots,n_k)})$ is a finite sequence of dense subsets of $X$, choose by Claim 1 a nonempty open set $U_\mu$ such that for each $x\in U_\mu$ there exists $D\in{\frak D}$ such that $x=\sigma(D_{(n_1)},\cdots, D_{(n_1,\cdots, n_k)},D)$. Choose a countable dense set $C_\mu\subset U_\mu$, and enumerate it as $\{x_{\mu\frown(n)}:n<\infty\}$. For each $n\in\omega$ choose a $D_{\mu\frown(n)}\in{\frak D}$ such that  $x_{\mu\frown(n)}=\sigma(D_{(n_1)}, \cdots, D_{(n_1,\cdots,n_k)}, D_{\mu\frown(n)})$.
This defines $C_\nu, x_\nu, U_\nu, D_\nu$ for all $\nu$ of length $k+1$ in $^{<\omega}\omega$.

{\flushleft{\bf Claim 2:}} $\{U_\mu:\mu\in\;^{<\omega}\omega\}$ is a $\theta$-$\pi$-base for $X$.\\
{\tt Proof of Claim 2: }
If not, choose a nonempty open set $V$ such that for all $\mu$, ${U_\mu}\nsubseteq\overline{V}$. Thus, for all $\mu$, $U_\mu\setminus\overline{V}$ is a nonempty open subset of $U_{\mu}$. 
Since $C_\emptyset$ is dense in $U_\emptyset$ choose $x_{(n_1)}\in C_\emptyset\setminus\overline{V}$, and since $C_{(n_1)}$ is dense in $U_{(n_1)}$ choose $x_{(n_1,n_2)}\in C_{(n_1)}\setminus\overline{V}$, and so on.

Then $D_{(n_1)}, x_{(n_1)}, D_{(n_1,n_2)}, x_{(n_1,n_2)}, \cdots$ is a $\sigma$-play of $\gone({\frak D},{\frak D}_\theta)$, but $\overline{V}\cap\{x_{(n_1)}, x_{(n_1,n_2)},\cdots\}=\emptyset$, so that TWO's selection is not $\theta$-dense. This  contradicts the fact that $\sigma$ is a winning strategy for TWO in the game $\gone(\mathfrak{D},\;\mathfrak{D}_{\theta})$. 
\end{proof}

It is not clear whether the condition that the space in question is separable, used in the proof of the implication $(4)\Rightarrow (1)$, is in general superfluous. There certainly are non-separable spaces in which TWO has a winning strategy in the game $\gone(\mathfrak{D}_{\theta},\mathfrak{D}^{gp}_{\theta})$:\\
For consider the space $(\reals,\,\tau_c)$.
For $U\in\tau_c$, the $\tau_c$ interior of the $\tau_c$ closure of $U$ coincides with the standard interior of the standard closure of $U$. It follows that the set of standard open intervals with rational endpoints is a countable $\theta$-$\pi$-base of $(\reals,\; \tau_c)$. For each $\tau_c$-open set $U$ and each $\theta$-dense set $D$, the intersection $D\cap \overline{U}$ is infinite.

By Theorem \ref{Separable}, TWO has a winning strategy in the game $\gone(\mathfrak{D}_{\theta},\mathfrak{D}^{gp}_{\theta})$ in the space $(\reals,\tau_c)$.

\bigskip
\bigskip

\section{Acknowledgements}

The research for the results reported in this paper partially occurred while Dr. Pansera was visiting the Department of Mathematics at Boise State University. The Department's hospitality and support during this visit is gratefully acknowledged.




\begin{thebibliography}{88}

\bibitem{AlexandroffUryshon} P. Alexandroff and P. Urysohn, \emph{M\'{e}moire sur les espaces topologiques compacts}, \textbf{Verhandelingen der koninklijke akademie van wetenschappen te Amsterdam, afdeeling natuurkunde (eerste sectie) deel XIV, No.1.} (1929), 1-96. 

\bibitem{AVA1} A.V. Arkhangelskii, \emph{Hurewicz spaces, analytic sets and fan tightness of function spaces}, {\bf Soviet Mathematical Doklady} 33 (1986), 396 Ð 399.

\bibitem{AVA} A.V. Arkhangelskii, \emph{Topological Function Spaces}, \textbf{Mathematics and Its Applications Series}, Kluwer Academic Publishers, Dordrecht, 1992.

\bibitem{LB} L. Babinkostova, \emph{On some questions about selective separability}, {\bf Mathematical Logic Quarterly} 55:3 (2009), 260 - 262.

\bibitem{BBM} A. Bella, M. Bonanzinga and M. Matveev, \emph{Variations of selective separability}, {\bf Topology and its Applications} 156 (2009), 1241 - 1252.

\bibitem{BonCamMatPan} M. Bonanzinga, F. Cammaroto, M. Matveev and B. Pansera, \emph{On Weaker Forms of Separability}, \textbf{Quaestiones Mathematicae} {\bf 31}(2008), 387-395.

\bibitem{BCPT} M. Bonanzinga, F. Cammaroto, B. Pansera and B. Tsaban, \emph{Diagonalizations of dense families}, {\bf Topology and its Applications} 165 (2014), 12 - 25.

\bibitem{CasDiMaio} A. Caserta and G. Di Maio, \emph{Variations on selective separability in non-regular spaces}, \textbf{Topology and its Applications} {160} (2013), 2379-2385.

\bibitem{GN} J. Gerlits and Zs. Nagy, \emph{Some properties of C(X)}, I, \textbf{Topology and its Applications} 14 (1982), 151--161.

\bibitem{WH} W. Hurewicz, \emph{\"Uber eine Verallgemeinerung des Borelschen Theorems} , {\bf Mathematische Zeitschrift} 24 (1925), 401 Ð 421.

\bibitem{COC2} W. Just, A. W. Miller, M. Scheepers and P. J. Szeptycki, \emph{The combinatorics of open covers (II)}, \textbf{Topology and its Applications} 73 (1996), 241--266

\bibitem{Sakai} M. Sakai, \emph{Property C" and function spaces}, {\bf Proceedings of the American Mathematical Society} 104 (1988), 917 - 919.

\bibitem{COC1} M. Scheepers, \emph{Combinatorics of open covers I: Ramsey theory}, \textbf{Topology and its Applications 69(1996)}, 31--62.

\bibitem{COC3} M. Scheepers, \emph{Combinatorics of open covers (III): $C_p(X)$ and games}, \textbf{Fundamenta Mathematicae 152 (1997)}, 231--254.

\bibitem{COC6} M. Scheepers, \emph{Combinatorics of open covers. VI. Selectors for sequences of dense sets},  \textbf{Quaestiones Mathematicae 22 (1999), no.1} 109--130.

\bibitem{Velichko} N.V. Velichko, \emph{H-closed topological spaces}, {\bf Mathematicheskii Zbornik} 70(112):1 (1966), 98 - 112. 

\end{thebibliography}
\end{document}